\documentclass[11pt,openany]{article}

\usepackage{amssymb}
\usepackage{amsmath}
\usepackage{amsthm}
\usepackage{amsfonts}
\usepackage{mathrsfs}
\usepackage{makeidx}
\usepackage{multicol}

\newtheorem{theorem}{Theorem}[section]
\newtheorem{corollary}[theorem]{Corollary}
\newtheorem{definition}[theorem]{Definition}
\newtheorem{lemma}[theorem]{Lemma}

\begin{document}
\title{\Huge \bf   Non-Archimedean meromorphic solutions of functional equations}
\author{ Pei-Chu Hu \& Yong-Zhi Luan$^*$ }
\date{}
\maketitle

\begin{abstract}
In this paper, we discuss meromorphic solutions of functional equations over non-Archimedean fields, and prove analogues of the Clunie lemma,  Malmquist-type theorem and Mokhon'ko theorem.
\end{abstract}

\pagenumbering{arabic}

\begin{figure}[b]
\rule[-2.5truemm]{5cm}{0.1truemm}\\[2mm] {\footnotesize
Mathematics Subject Classification 2000 (MSC2000). Primary 11S80, 12H25; Secondary 30D35. \\ $^*$The work of first
author was partially supported by National Natural Science Foundation of China (Grant No. 11271227), and supported partially by PCSIRT ( IRT1264).}
\end{figure}


\section{Introduction}

Let $\kappa$ be an algebraically closed field of characteristic zero,
complete for a non-trivial non-Archimedean absolute value $|\cdot|$.  Let $\mathcal{A}(\kappa)$ (resp. $\mathcal{M}(\kappa)$) denote the set of
entire (resp. meromorphic) functions over $\kappa$. As usual, if $R$ is a ring, we use $R[X_0,X_1,...,X_n]$ to denote the ring of polynomials of variables $X_0,X_1,...,X_n$ over $R$. We will use the following assumption:
\begin{description}
  \item[(A)]  Fix a positive integer $n$. Take $a_i,b_i$ in $\kappa$ such that $|a_i|=1$ for each $i=0,1,...,n$, and such that
  \begin{equation*}
    L_i(z)=a_iz+b_i\ (i=0,1,...,n)
  \end{equation*}
  are distinct, where $a_0=1,b_0=0$. Let $f$ be a non-constant meromorphic function over $\kappa$ and write
  \begin{equation*}
    f_i=f\circ L_i,\ i=0,1,...,n
  \end{equation*}
  with $f_0=f$. Take non-zero elements
\begin{equation*}
B\in \mathcal{M}(\kappa)[X];\ \Omega,\Phi\in\mathcal{M}(\kappa)[X_0,X_1,...,X_n].
\end{equation*}
\end{description}

Under the assumption (A), there exist $\{b_0,...,b_q\}\subset {\cal M}(\kappa)$ with $b_q\not\equiv 0$
such that
\begin{equation}\label{Bfunc}
B(X)=\sum_{k=0}^qb_kX^k.
\end{equation}
Similarly, write
\begin{equation}\label{diffop}
\Omega\left(X_0,X_1,...,X_n\right)=\sum_{i\in I}c_iX_0^{i_0}X_1^{i_1}\cdots
X_n^{i_n},
\end{equation}
where $i=(i_0,i_1,...,i_n)$ are non-negative integer indices, $I$ is a
finite set, $c_i\in {\cal M}(\kappa)$, and
\begin{equation}
\Phi\left(X_0,X_1,...,X_n\right)=\sum_{j\in J}d_jX_0^{j_0}X_1^{j_1}\cdots
X_n^{j_n},\label{diffopII}
\end{equation}
where $j=(j_0,j_1,...,j_n)$ are non-negative integer indices, $J$ is a
finite set, $d_j\in {\cal M}(\kappa)$.

In this paper, we will use the symbols from \cite{HY} on value distribution of meromorphic functions.
For example, let $\mu(r,f)$ denote the maximum term of power series for $f\in \mathcal{A}(\kappa)$ and its fractional extension to $\mathcal{M}(\kappa)$,
$m(r,f)$ the compensation (or proximity) function of $f$, $N(r,f)$ the valence function of $f$ for poles, and the characteristic function of $f$
\begin{equation*}
    T(r,f)=m(r,f)+N(r,f).
\end{equation*}
 Now we can state our results as follows:

\begin{theorem}\label{Clunie}
Assume that the condition $({\rm A})$ holds.
If $f$ is a solution of the following functional equation
\begin{equation}\label{Clun-dif-eq}
    B(f)\Omega(f,f_1,...,f_n)=\Phi(f,f_1,...,f_n)
\end{equation}
with $\deg B\geq \deg \Phi$, then
\begin{equation}\label{Clunieeq}
m(r,\Omega)\leq \sum_{i\in I}m(r,c_i)+
\sum_{j\in J}m(r,d_j)+lm\left(r,\frac{1}{b_q}\right)+
l\sum_{j=0}^qm(r,b_j),
\end{equation}
where $l=\max\{1,\deg\Omega\}$, $\Omega=\Omega(f,f_1,...,f_n)$. Further, if $\Phi$ is a polynomial of $f$, we also have
\begin{equation}\label{Clunieeq1}
N(r,\Omega)\leq \sum_{i\in I}N(r,c_i)+
\sum_{j\in J}N(r,d_j)+O\left(\sum_{j=0}^qN\left(r,\frac{1}{b_j}\right)\right).
\end{equation}
\end{theorem}

Theorem~\ref{Clunie} is a difference analogue of the Clunie lemma over non-Archimedean fields (cf. \cite{HY}) .
R. G. Halburd and R. J. Korhonen \cite{HK} obtained a difference analogue of  the Clunie lemma over the field of complex numbers (cf. \cite{Clunie}).
Theorem~\ref{Clunie} has numerous applications in the study of non-Archimedean difference equations, and beyond. To state one of its applications, we need the following notation:

\begin{definition} A solution $f$ of $(\ref{Clun-dif-eq})$ is said to be admissible\label{admissible-sol} if
$f\in {\cal M}(\kappa)$ satisfies $(\ref{Clun-dif-eq})$ with
\begin{equation}\label{admis-def}
\sum_{i\in I}T(r,c_i)+
\sum_{j\in J}T(r,d_j)+\sum_{k=0}^qT(r,b_k)=o(T(r,f)),
\end{equation}
equivalently, the coefficients  of $B,\Phi,\Omega$ are slowly moving targets with respect to $f$.
\end{definition}

If all $c_i,d_j,b_k$ are rational functions, each transcendental meromorphic function $f$ over $\kappa$ must satisfy (\ref{admis-def}), which means that each transcendental meromorphic solution $f$ over $\kappa$ is admissible.

\begin{theorem}\label{Malmth}
If $\Phi$ is of the form
\begin{equation*}
    \Phi(f,f_1,...,f_n)=\Phi(f)=\sum_{j=0}^pd_jf^j,
\end{equation*}
and if
$(\ref{Clun-dif-eq})$ has an admissible non-constant meromorphic solution $f$, then
$$q=0, \quad p\leq \deg(\Omega).$$
\end{theorem}

Theorem~\ref{Malmth}  is a difference analogue of a Malmquist-type theorem over non-Archimedean fields (cf. \cite{HY}) .
Malmquist-type theorems were obtained by Malmquist \cite{Ma}, Gackstatter-Laine \cite{GL}, Laine \cite{L1},
Toda \cite{To}, Yosida \cite{Yo} (or see He-Xiao \cite{HX}) for meromorphic functions on $\mathbb{C}$, and  Hu-Yang \cite{HY5} or \cite{HY6} for several complex variables.

\begin{corollary} \label{Malmth-cor}
Assume that the condition $({\rm A})$ holds such that the coefficients of $B,\Omega,\Phi$ are rational functions over $\kappa$, and such that $\Phi$ has the form in Theorem~\ref{Malmth} .
If $(\ref{Clun-dif-eq})$  has a transcendental meromorphic
solution $f$ over $\kappa$, then $\Phi/B$ is a polynomial
in $f$ of degree $\leq \deg(\Omega)$.
\end{corollary}

Corollary~\ref{Malmth-cor} is a difference analogue of
the non-Archimedean Malmquist-type theorem due to Boutabaa \cite{Boutabaa}.

\begin{theorem}\label{Mokho-thm}
Let $f\in \mathcal{M}(\kappa)$ be a non-constant admissible solution of
\begin{equation}\label{Mokho-eq}
    \Omega\left(f,f',...,f^{(n)}\right)=0,
\end{equation}
where the solution $f$ is called admissible if
$$\sum_{i\in I}T(r,c_i)=o(T(r,f)).$$
If a slowly moving target $a\in \mathcal{M}(\kappa)$ with respect to $f$, that is,
$$T(r,a)=o(T(r,f)),$$
does not satisfy the equation $(\ref{Mokho-eq})$, then
\begin{equation*}
    m\left(r,\frac{1}{f-a}\right)=o(T(r,f)).
\end{equation*}
\end{theorem}

Theorem~\ref{Mokho-thm} is an analogue of a result due to A. Z. Mokhon'ko and  V. D. Mokhon'ko \cite{MM} over non-Archimedean fields, which also has a difference analogue as follows:

\begin{theorem}\label{Mokho-FEthm}
Assume that the condition $({\rm A})$ holds. Let $f\in \mathcal{M}(\kappa)$ be a non-constant admissible solution of
\begin{equation}\label{Mokho-dif-eq}
    \Omega\left(f,f_1,...,f_n\right)=0,
\end{equation}
where the solution $f$ is called admissible if
$$\sum_{i\in I}T(r,c_i)=o(T(r,f)).$$
If a slowly moving target $a\in \mathcal{M}(\kappa)$ with respect to $f$
does not satisfy the equation $(\ref{Mokho-dif-eq})$, then
\begin{equation*}
    m\left(r,\frac{1}{f-a}\right)=o(T(r,f)).
\end{equation*}
\end{theorem}

A version of Theorem~\ref{Mokho-FEthm} over complex number field can be found in \cite{HK}.

\section{Difference analogue of the Lemma on the Logarithmic Derivative}

Take $a(\not=0),b\in\kappa$ and consider the linear transformation
\begin{equation*}
    L(z)=az+b
\end{equation*}
over $\kappa$. For a positive integer $m$, set
\begin{equation*}
    \Delta_Lf=f\circ L-f,\ \Delta^m_Lf=\Delta_L(\Delta_L^{m-1}f).
\end{equation*}

\begin{lemma}\label{LLD-ent}
Take $f\in \mathcal{A}(\kappa)$ and assume $|a|\leq 1$. When $r>|b|/|a|$, we have
\begin{equation*}
    \mu(r,f\circ L)\leq \mu(r,f).
\end{equation*}
Moreover, we obtain
\begin{equation*}
    \mu\left(r,\frac{f\circ L}{f}\right) \leq 1,\ \mu\left(r,\frac{\Delta_L^mf}{f}\right) \leq 1.
\end{equation*}
\end{lemma}

\begin{proof}
We can write
\begin{equation*}
    f(z)=\sum_{n=0}^\infty a_nz^n
\end{equation*}
since $f\in \mathcal{A}(\kappa)$. Therefore
\begin{equation*}
    f(L(z))=\sum_{n=0}^\infty a_n(az+b)^n.
\end{equation*}

First of all, we take $r\in|\kappa|$, that is, $r=|z|$ for some $z\in \kappa$. When $r>|b|/|a|$, we find (cf. \cite{HY})
\begin{equation*}
    \mu(r,f\circ L)=|f(L(z))|\leq \max_{n\geq 0}|a_n||az+b|^n=\max_{n\geq 0}|a_n||az|^n\leq \max_{n\geq 0}|a_n||z|^n=\mu(r,f).
\end{equation*}
In particular,
\begin{equation*}
    \mu\left(r,\frac{f\circ L}{f}\right)=\frac{ \mu(r,f\circ L)}{\mu(r,f)} \leq 1,
\end{equation*}
and hence
\begin{equation*}
    \mu\left(r,\frac{\Delta_Lf}{f}\right)=\frac{ \mu(r,f\circ L-f)}{\mu(r,f)}\leq\frac{1}{\mu(r,f)} \max\{ \mu(r,f\circ L), \mu(r,f)\}\leq 1.
\end{equation*}
By induction, we can prove
\begin{equation*}
    \mu\left(r,\frac{\Delta_L^mf}{f}\right)\leq 1.
\end{equation*}
Since $|\kappa|$ is dense in $\mathbb{R}_+=[0,\infty)$, by using continuity we easily see that these inequalities hold for all $r>|b|/|a|$.
\end{proof}

Note that (cf. \cite{HY})
\begin{equation}
    m(r,f)=\log^+\mu(r,f)=\max\{0,\log\mu(r,f)\}.
\end{equation}
Lemma~\ref{LLD-ent} implies immediately the following difference analogue of the Lemma on the Logarithmic Derivative:

\begin{corollary}
Take $f\in \mathcal{A}(\kappa)$ and assume $|a|\leq 1$. When $r>|b|/|a|$, we have
\begin{equation*}
    m\left(r,\frac{f\circ L}{f}\right) =0,\ m\left(r,\frac{\Delta_L^mf}{f}\right) =0.
\end{equation*}
\end{corollary}

\begin{lemma}\label{LLD-mer}
Take $f\in \mathcal{M}(\kappa)-\{0\}$ and assume $|a|=1$. When $r>|b|$, we have
\begin{equation}\label{Inv-form}
    \mu(r,f\circ L)= \mu(r,f).
\end{equation}
Moreover, we obtain
\begin{equation*}
    \mu\left(r,\frac{f\circ L}{f}\right) = 1,\ \mu\left(r,\frac{\Delta_L^mf}{f}\right) \leq 1.
\end{equation*}
\end{lemma}

\begin{proof}
Since $f\in \mathcal{M}(\kappa)-\{0\}$, there are $g,h(\not=0)\in \mathcal{A}(\kappa)$ with $f=\frac{g}{h}$. Thus   (cf. \cite{HY})
\begin{equation}\label{Inv-form2}
    \mu(r,f\circ L)=\frac{\mu(r,g\circ L)}{\mu(r,h\circ L)}.
\end{equation}
Take $r\in|\kappa|$. Since $|a|=1$, we have
\begin{equation*}
    |L(z)|=|az+b|=|z|=r
\end{equation*}
when $r>|b|$, and so
\begin{equation*}
    \mu(r,g\circ L)=\mu(r,g).
\end{equation*}
Similarly, we have $\mu(r,h\circ L)=\mu(r,h)$. Thus the formula (\ref{Inv-form}) holds.  By using continuity we easily see that the inequality holds for all $r>|b|$.
\end{proof}

\begin{corollary}
Take $f\in \mathcal{M}(\kappa)-\{0\}$ and assume $|a|=1$. When $r>|b|$, we have
\begin{equation*}
    m\left(r,\frac{f\circ L}{f}\right) = 0,\ m\left(r,\frac{\Delta_L^mf}{f}\right) =0.
\end{equation*}
\end{corollary}

\section{Proof of Theorem~\ref{Clunie}}

To prove (\ref{Clunieeq}), take $z\in\kappa$ with
$$f(z)\not=0,\infty;\quad b_k(z)\not=0,\infty\quad (0\leq k\leq q); $$
$$ c_i(z)\not=0,\infty\quad (i\in I);
\quad d_j(z)\not=0,\infty\quad (j\in J).$$
Write
$$b(z)=\max_{0\leq k< q}\left\{1,\left(\frac{|b_k(z)|}{|b_q(z)|}\right)^{
\frac{1}{q-k}}\right\}.$$
If $|f(z)|>b(z)$, we have
$$|b_k(z)||f(z)|^k\leq |b_q(z)|b(z)^{q-k}|f(z)|^k
<|b_q(z)||f(z)|^q,$$
and hence
$$|B(f)(z)|=|b_q(z)||f(z)|^q.$$
Then
$$|\Omega\left(f,f_1,...,f_n\right)(z)|=\frac{|\Phi(f,f_1,...,f_n)(z)|}{|B(f)(z)|}
\leq \frac{1}{|b_q(z)|}\max_{j\in J}|d_j(z)|\left|\frac{f_1(z)}{f(z)}\right|^{j_1}\cdots\left|\frac{f_n(z)}{f(z)}\right|^{j_n} .$$
If $|f(z)|\leq b(z)$,
$$|\Omega(f,f_1,...,f_n)(z)|\leq b(z)^{\deg(\Omega)}\max_{i\in I}
|c_i(z)|\left|\frac{f_1(z)}{f(z)}\right|^{i_1}\cdots
\left|\frac{f_n(z)}{f(z)}\right|^{i_n}.$$
Therefore, in any case, the inequality
\begin{eqnarray*}
\mu(r,\Omega)&\leq& \max_{j\in J,i\in I}\left\{\frac{\mu(r,d_j)}
{\mu(r,b_q)}\prod_{k=1}^n\mu\left(r,\frac{f_k}{f}\right)^{j_k},\mu(r,c_i)\prod_{k=1}^n\mu\left(r,\frac{f_k}{f}\right)^{i_k}\right.\\
& & \cdot\max_{0\leq k<q}
\left.\left\{1,\mu\left(r,\frac{b_k}{b_q}\right)^{\frac{\deg(\Omega)}{q-k}}
\right\}\right\}
\end{eqnarray*}
holds where $r=|z|$,
which also holds for all $r>0$ by continuity of the functions $\mu$. By using Lemma~\ref{LLD-mer}, we find
\begin{equation*}
 \mu(r,\Omega)\leq    \max_{j\in J,i\in I}\left\{\frac{\mu(r,d_j)}{\mu(r,b_q)}, \mu(r,c_i) \cdot\max_{0\leq k<q}
\left\{1,\mu\left(r,\frac{b_k}{b_q}\right)^{\frac{\deg(\Omega)}{q-k}}
\right\}\right\},
\end{equation*}
and hence (\ref{Clunieeq}) follows from this inequality.

According to the proof of (4.9) in \cite{HY}, we easily obtain the inequality (\ref{Clunieeq1}).

\section{Proof of Theorem~\ref{Malmth} }

By using the algorithm of division, we have
$$\Phi(f)=\Phi_1(f)B(f)+\Phi_2(f)$$
with $\deg(\Phi_2)<q$.
Thus, the equation $(\ref{Clun-dif-eq})$ can be rewritten as follows:
\begin{equation}\label{New-Alg-div-eq}
\Omega\left(f,f_1,...,f_n\right)-\Phi_1(f)=\frac{\Phi_2(f)}{B(f)}.
\end{equation}
Applying Theorem~\ref{Clunie} to this equation, we obtain
$$m(r,\Omega-\Phi_1)=o(T(r,f)),$$
\begin{equation*}
N(r,\Omega-\Phi_1)=o(T(r,f)),
\end{equation*}
and hence
\begin{equation*}
    T(r,\Omega-\Phi_1)=o(T(r,f)).
\end{equation*}
Theorem 2.12 due to Hu-Yang \cite{HY} implies
$$T(r,\Omega-\Phi_1)=T\left(r,\frac{\Phi_2}{B}\right)=qT(r,f)+o(T(r,f)).$$
It follows that $q=0$, and $(\ref{Clun-dif-eq})$ assumes the following
form
$$\Omega\left(f,f_1,...,f_n\right)=\Phi(f).$$
Thus, Theorem 2.12 in \cite{HY} implies
\begin{equation}\label{Opestim}
T(r,\Omega)=T(r,\Phi)=pT(r,f)+o(T(r,f)).
\end{equation}

On other hand, it is easy to find the following eastimate
\begin{equation}
N(r,\Omega)\leq \deg(\Omega)N(r,f)+\sum_{i\in I}N(r,c_i).
\end{equation}
Obviously, we also have
\begin{equation}
m(r,\Omega)\leq \deg(\Omega)m(r,f)
+\max_{i\in I}\left\{m(r,c_i)+\sum_{\alpha=1}^ni_{\alpha}m\left(r,
\frac{f_\alpha}{f}\right)\right\}.
\end{equation}
By Lemma~\ref{LLD-mer}, we obtain
\begin{equation}\label{Opestim1}
T(r,\Omega)\leq \deg(\Omega)T(r,f)
+\sum_{i\in I}T(r,c_i)+O(1).
\end{equation}
Our result follows from  (\ref{Opestim}) and (\ref{Opestim1}).

\section{Proof of Theorem~\ref{Mokho-thm} ,~\ref{Mokho-FEthm}}

By substituting $f=g+a$ into (\ref{Mokho-eq}), we obtain
\begin{equation*}
    \Psi+P=0,
\end{equation*}
where
\begin{equation*}
    \Psi\left(g,g',...,g^{(n)}\right)=\sum_iC_ig^{i_0}(g')^{i_1}\cdots(g^{(n)})^{i_n}
\end{equation*}
is a differential polynomial of $g$ such that all of its terms are at least of degree one, and
$$T(r,P)=o(T(r,f)).$$
Also $P\not\equiv 0$, since $a$ does not satisfy  (\ref{Mokho-eq}).

Take $z\in\kappa$ with
$$g(z)\not=0,\infty;\ C_i(z)\not=\infty;\ P(z)\not=0,\infty.$$
Set $r=|z|$. If $|g(z)|\geq 1$, then
\begin{equation*}
    m\left(r,\frac{1}{g}\right)=\max\left\{0,\log\frac{1}{|g(z)|}\right\}=0.
\end{equation*}
It is therefore sufficient to consider only the case $|g(z)|< 1$. But then,
\begin{eqnarray*}
  \left|\frac{\Psi\left(g(z),g'(z),...,g^{(n)}(z)\right)}{g(z)}\right| &=& \frac{1}{|g(z)|} \left|\sum_iC_i(z)g(z)^{i_0}g'(z)^{i_1}\cdots g^{(n)}(z)^{i_n}\right|\\
  &\leq& \max_i|C_i(z)|\left|\frac{g'(z)}{g(z)}\right|^{i_1}\cdots \left|\frac{g^{(n)}(z)}{g(z)}\right|^{i_n}
\end{eqnarray*}
since $i_0+\cdots i_n\geq 1$ for all $i$. Therefore,
\begin{eqnarray*}
  m\left(r,\frac{1}{g}\right) &=& \log \frac{1}{|g(z)|}=\log \frac{|P(z)|}{|g(z)|} +\log \frac{1}{|P(z)|} \\
  &=& \log \frac{|\Psi\left(g(z),g'(z),...,g^{(n)}(z)\right)|}{|g(z)|} +\log \frac{1}{|P(z)|}\\
  &\leq& \sum_i\left\{m(r,C_i)+i_1m\left(r,\frac{g'}{g}\right) +\cdots +i_nm\left(r,\frac{g^{(n)}}{g}\right) \right\} +m\left(r,\frac{1}{P}\right) \\
  &=&o(T(r,f)).
\end{eqnarray*}
Since $g=f-a$, the assertion follows.

Obviously, according to the method above, we can prove Theorem~\ref{Mokho-FEthm} similarly.

\section{Final notes}

We will use the following assumption:
\begin{description}
  \item[(B)]  Fix a positive integer $n$. Take $a_i,b_i$ in $\kappa$ such that $|a_i|=1$ for each $i=1,...,n$, and such that
  \begin{equation*}
    L_i(z)=a_iz+b_i\ (i=1,...,n)
  \end{equation*}
  satisfy $L_i(z)\not=z$ for each $i=1,...,n$. Let $f$ be a non-constant meromorphic function over $\kappa$ and let $\{f_1,...,f_m\}$ be a finite set consisting of the forms $\Delta_{ L_i}^jf$. Take
  \begin{equation*}
B\in \mathcal{M}(\kappa)[f];\ \Omega,\Phi\in\mathcal{M}(\kappa)[f,f_1,...,f_m].
\end{equation*}
\end{description}

According to the methods in this paper, we can prove easily the following results:

\begin{theorem}\label{Clunie-dif}
Assume that the condition $({\rm B})$ holds.
If $f$ is a solution of the following equation
\begin{equation}\label{Clun-dif-eq2}
    B(f)\Omega(f,f_1,...,f_m)=\Phi(f,f_1,...,f_m)
\end{equation}
with $\deg B\geq \deg \Phi$, then
\begin{equation}\label{Clunieeq2}
m(r,\Omega)\leq \sum_{i\in I}m(r,c_i)+
\sum_{j\in J}m(r,d_j)+lm\left(r,\frac{1}{b_q}\right)+
l\sum_{j=0}^qm(r,b_j),
\end{equation}
where $l=\max\{1,\deg\Omega\}$, $\Omega=\Omega(f,f_1,...,f_m)$. Further, if $\Phi$ is a polynomial of $f$, we also have
\begin{equation}\label{Clunieeq-N2}
N(r,\Omega)\leq \sum_{i\in I}N(r,c_i)+
\sum_{j\in J}N(r,d_j)+O\left(\sum_{j=0}^qN\left(r,\frac{1}{b_j}\right)\right).
\end{equation}
\end{theorem}

\begin{theorem}\label{Malmth-dif}
If $\Phi$ is of the form
\begin{equation*}
    \Phi(f,f_1,...,f_m)=\Phi(f)=\sum_{j=0}^pd_jf^j,
\end{equation*}
and if
$(\ref{Clun-dif-eq2})$ has an admissible non-constant meromorphic solution $f$, then
$$q=0, \quad p\leq \deg(\Omega).$$
\end{theorem}

\begin{theorem}
Assume that the condition $({\rm B})$ holds. Let $f\in \mathcal{M}(\kappa)$ be a non-constant admissible solution of
\begin{equation}\label{Mokho-dif2-eq}
    \Omega\left(f,f_1,...,f_m\right)=0,
\end{equation}
where the solution $f$ is called admissible if
$$\sum_{i\in I}T(r,c_i)=o(T(r,f)).$$
If a slowly moving target $a\in \mathcal{M}(\kappa)$ with respect to $f$
does not satisfy the equation $(\ref{Mokho-dif2-eq})$, then
\begin{equation*}
    m\left(r,\frac{1}{f-a}\right)=o(T(r,f)).
\end{equation*}
\end{theorem}

School of Mathematics \par Shandong University\par Jinan
250100, Shandong, China\par E-mail: pchu@sdu.edu.cn\par E-mail: luanyongzhi@gmail.com

\end{document}